\documentclass[final]{dmtcs-episciences}

\usepackage[utf8]{inputenc}
\usepackage[T1]{fontenc}
\usepackage{subfigure}
\usepackage{latexsym,amssymb,amsfonts,amsmath}
\usepackage{amsthm}
\usepackage[all]{xy}
\usepackage{float}
\usepackage{hyperref}

\theoremstyle{plain}
\newtheorem{proposition}{Proposition}
\newtheorem{lemma}[proposition]{Lemma}
\newtheorem{theorem}[proposition]{Theorem}
\newtheorem{corollary}[proposition]{Corollary}

\theoremstyle{definition}
\newtheorem{definition}[proposition]{Definition}

\theoremstyle{remark}

\usepackage{graphicx}

\usepackage[round]{natbib}

\author{Alejandro Contreras-Balbuena\affiliationmark{1,3}\thanks{Corresponding author.}
  \and Hortensia Galeana-S\'{a}nchez\affiliationmark{1}
  \and Ilan A. Goldfeder\affiliationmark{1,2}}
\title{Alternating Hamiltonian cycles in $2$-edge-colored multigraphs\thanks{This research was supported by grants UNAM-DGAPA-PAPIIT IN104717 and CONACyT CB-2013/219840. The first and third authors were supported by postdoctoral positions under the grant CONACyT CB-2013/219840 at Instituto de Matemáticas at the Universidad Nacional Autónoma de México.}}
\affiliation{
  Instituto de Matem\'{a}ticas, Universidad Nacional Aut\'{o}noma de M\'{e}xico, M\'{e}xico\\
  Departamento de Matem\'{a}ticas, Universidad Aut\'{o}noma Metropolitana, Iztapalapa, M\'{e}xico\\
  Universidad del Istmo, Campus Tehuantepec, M\'{e}xico}
\received{2018-10-11}
\revised{2019-04-27, 2019-05-18}
\accepted{2019-05-28}
\begin{document}
\publicationdetails{21}{2019}{1}{12}{4882}
\maketitle
\begin{abstract}
A path (cycle) in a $2$-edge-colored multigraph is alternating if no two consecutive edges have the same color. The problem of determining the existence of alternating Hamiltonian paths and cycles in $2$-edge-colored multigraphs is an $NP$-complete problem and it has been studied by several authors. In Bang-Jensen and Gutin's book \textit{Digraphs: Theory, Algorithms and Applications}, it is devoted one chapter to survey the latest results on this topic. Most results on the existence of alternating Hamiltonian paths and cycles concern on complete and bipartite complete multigraphs and a few ones on multigraphs with high monochromatic degrees or regular monochromatic subgraphs. In this work, we use a different approach imposing local conditions on the multigraphs and it is worthwhile to notice that the class of multigraphs we deal with is much larger than, and includes, complete multigraphs, and we provide a full characterization of this class.

Given a $2$-edge-colored multigraph $G$, we say that $G$ is \emph{$2$-$\mathcal{M}$-closed} (resp. \emph{$2$-$\mathcal{NM}$-closed)} if for every monochromatic (resp. non-monochromatic) $2$-path $P=(x_1$, $x_2$, $x_3)$, there exists an edge between $x_1$ and $x_3$. In this work we provide the following characterization:
\begin{quote}
A $2$-$\mathcal{M}$-closed multigraph has an alternating Hamiltonian cycle if and only if it is color-connected and it has an alternating cycle factor.
\end{quote}
Furthermore, we construct an infinite family of $2$-$\mathcal{NM}$-closed graphs, color-connected, with an alternating cycle factor, and with no alternating Hamiltonian cycle.
\end{abstract}

A trail $T$ in a $2$-edge-colored multigraph is \emph{alternating} if no two consecutive edges of $T$ have the same color. A great number of problems have been modelled by using edge-colored multigraphs. \cite{BanGut97} suggested that \cite{Pet1891}’s paper seems to be the first place where one can find applications of alternating colored trails. Moreover, there are applications of alternating colored trails in genetics, \cite{AliKarNew95, Dor87, Dor94, DorTim87, Pev95}, transportation and connectivity problems, \cite{GouLyrMar2010, WirSte2010}, social sciences, \cite{ChoManMeg93}, and graphs models for conflicts resolutions, \cite{XuKilHip2010, XuLiHip2009, XuLiKil2009}.

The problem of determining the existence of alternating Hamiltonian paths and cycles in $2$-edge-colored multigraphs is $NP$-complete (Section~16.7 of \cite{BanGut2009}). There is an extensive literature on alternating colored trails. Chapter~16 of~\cite{BanGut2009} is devoted to the topic, there is a survey in~\cite{KanLi2008} and more recent papers include~\cite{AboDasFar2008, AboDasFer2010, GorPop2012, GouLyrMar2012, GutJonShe2017a, GutJonShe2017b} and~\cite{Lo2014a, Lo2014b, Lo2016}.

It is worthwhile to mention the relationship between directed graphs and $2$-edge-colored multigraphs for Hamiltonian characterizations. For $2$-edge-colored multigraphs, two key ingredients are alternating cycle-factors and color-connectivity. These are the corresponding colored notions of cycle factors and strong connectivity for digraphs. In fact, many notions and results on alternating paths and cycles in $2$-edge-colored are analogues of similar but earlier results on digraphs. In particular, \cite{Gut84} was the first to characterize Hamiltonian bipartite tournaments using strong connectivity and cycle factors. \cite{Saa96} was the first to use alternating cycle factors and color-connectivity and he proved the same result as the authors but only for $2$-edge-colored complete multigraphs. On the other way, results on $2$-edge-colored multigraphs had been used to prove results on digraphs. For example, \cite{HagMan89} also characterized Hamiltonian bipartite tournaments, but they used a result on Hamiltonian cycles in $2$-edge-colored complete graphs by \cite{BanBan68}.

In~\cite{ConGalGol2017}, we gave a sufficient condition for a $2$-edge-colored multigraph to possess an alternating Hamiltonian cycle. Let $G$ be a $2$-edge-colored multigraph. An alternating $3$-path $P=(x_1$, $x_2$, $x_3$, $x_4)$ is \emph{closed-alternating} if there exist $y$, $w\in V(G)$ such that $C=(x_1$, $y$, $w$, $x_4$, $x_1)$ is an alternating cycle. Moreover, if every alternating $3$-path in $G$ is closed-alternanting, then we say that $G$ is \emph{closed-alternating}. We proved the following result.

\begin{theorem}[\cite{ConGalGol2017}]
\label{ConGalGol2017}
 Let $G$ be a connected $2$-edge-colored multigraph. If $G$ is closed-alternating and it has an alternating cycle factor, then $G$ has an alternating Hamiltonian cycle.
\end{theorem}

\section{Notation and terminology}
For terminology and notation not defined here, we refer the reader to Chapter~16 of \cite{BanGut2009}. In this paper, $G = (V (G)$, $E(G))$ denotes a loopless multigraph. A \emph{$c$-edge-coloring} of $G$ is a function $c\colon E(G) \to \{1$, \ldots, $c\}$ and given an edge $[u,v]$ in $G$, its color is denoted by $c[u$, $v]$. We sometimes refer to $c$-edge-coloring as an \emph{edge-coloring}. For $2$-edge-colored multigraphs, we use colors blue and red instead of $1$ and $2$. In our figures, the solid edges are blue, the dashed edges are red and the  dash-and-dotted edges are edges with unknown color.  Throughout this paper we only consider $2$-edge-colored multigraphs.

A path (cycle) in a $2$-edge-colored multigraph is \emph{alternating} if no two consecutive edges have the same color. A path (cycle) is \emph{Hamiltonian} if it visits every vertex in the multigraph. A \emph{cycle factor} is a collection of mutually vertex-disjoint cycles that cover all vertices.

Let $G$ be a $2$-edge-colored multigraph, $G$ is \emph{color-connected} if for every pair of distinct vertices $x$, $y$ in $G$, there exist alternating $(x$, $y$)-paths $P$ and $Q$ such that their first edges have different color and their last edges have different color.

\begin{definition}
Let $G$ be a $2$-edge-colored multigraph. We say that $G$ is \emph{$2$-$\mathcal{M}$-closed} (resp. \emph{$2$-$\mathcal{NM}$-closed}) if for every monochromatic (resp. non-monochromatic) $2$-path $P=(x_1$, $x_2$, $x_3)$, there exists an edge between $x_1$ and $x_3$.
\end{definition}

\section{Main result}
In this work we provide the following characterization:
\begin{theorem}
\label{teo:M-closed-characterization}
Let $G$ be a $2$-$\mathcal{M}$-closed multigraph. $G$ has an alternating Hamiltonian cycle if and only if $G$  is color-connected and has an alternating cycle factor.
\end{theorem}
We also provide an infinite family of $2$-$\mathcal{NM}$-closed graphs, color-connected, with an alternating cycle factor, and with no alternating Hamiltonian cycle at all.

\begin{definition}
Let $G$ be a $2$-edge-colored multigraph. Given two alternating cycles $C_1=(x_1$, \ldots, $x_n$, $x_1)$, $C_2=(y_1$, \ldots, $y_m$, $y_1)$ and an edge $[x_i$, $y_j]$ between the cycles, we say that $C_1$ and $C_2$ are \emph{appropriately labelled with respect to $[x_i$, $y_j]$} whenever $c[x_i$, $x_{i+1}]=c[y_j$, $y_{j+1}]=c[x_i$, $y_j]$ (see Figure~\ref{a2}).
\end{definition}
\begin{figure}
  \centerline{\includegraphics[height=3cm]{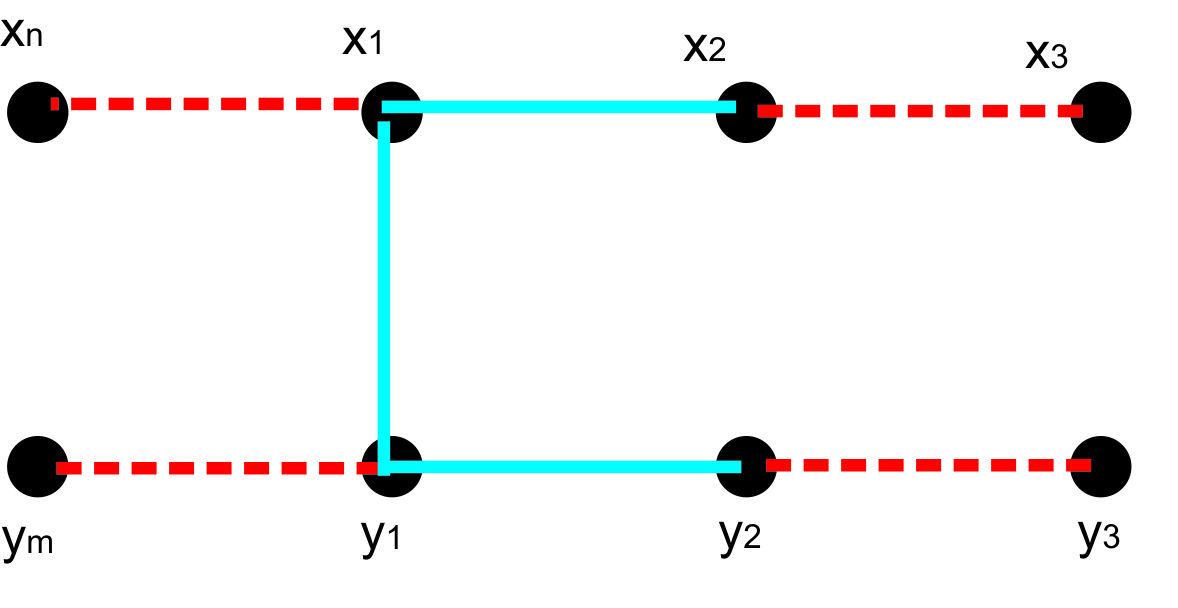}}
  \caption{Appropriately labelled cycles, the solid edges are blue and the dashed edges are red}
\label{a2}
\end{figure}

In~\cite{ConGalGol2017}, we focused on certain pair of edges which allow us to merge two alternating cycles into one alternating cycle. 
\begin{proposition}
\label{prop:goodpair}
Let $C_1=(x_1$, \ldots, $x_n$, $x_1)$ and $C_2=(y_1$, \ldots, $y_m$, $y_1)$ be two alternating cycles in a $2$-edge-colored multigraph $G$. If there exist two edges $[x_i$, $x_{i+1}]$, $[y_j$, $y_{j+1}]$ such that $(x_i$, $y_j$, $y_{j+1}$, $x_{i+1}$, $x_i)$ or $(x_i$, $y_{j+1}$, $y_{j}$, $x_{i+1}$, $x_i)$ is a monochromatic $4$-cycle, then there exists an alternating cycle with vertex set $V(C_{1})\cup V(C_{2})$. We call $[x_i$, $x_{i+1}]$, $[y_j$, $y_{j+1}]$ a \emph{good pair of edges}.
\end{proposition}

\begin{proof}
If $(x_i$, $y_j$, $y_{j+1}$, $x_{i+1}$, $x_i)$ (respectively $(x_i$, $y_{j+1}$, $y_{j}$, $x_{i+1}$, $x_i)$) is monochromatic, then the alternating cycle $(x_i$, $y_j$, $y_{j-1}$, \ldots, $y_{j+2}$, $y_{j+1}$, $x_{i+1}$, $x_{i+2}$, \ldots, $x_{i-1}$, $x_{i})$ (resp. $(x_i$, $y_{j+1}$, $y_{j+2}$, \ldots, $y_{j-1}$, $y_{j}$, $x_{i+1}$, $x_{i+2}$, \ldots, $x_{i-1}$, $x_{i})$) merges both cycles.
\end{proof}

We will use good pairs of edges extensively throughout this work.

Now, we describe the structure of the arcs between two alternating cycles. Let $C_1=(x_1$, \ldots, $x_n$, $x_1)$ and $C_2=(y_1$, \ldots, $y_m$, $y_1)$ be two alternating cycles in a $2$-$\mathcal{M}$-closed multigraph. If there exists an  edge $[x_i$, $y_j]$ such that $c[x_i$, $x_{i+1}]=c[y_j$, $y_{j+1}]=c[x_i$, $y_j]$ and there are no good pairs of edges between $C_1$ and $C_2$, then we prove that there exists $[x_{i+1}$, $y_{j+1}]$ in the following lemma. Since there are no good pairs of edges, $[x_{i+1}$, $y_{j+1}]$ and $[x_{i}$, $y_{j}]$ have different color. Repeating this procedure, we have all the edges of the form $[x_{i+k},y_{j+k}]$ for any $k\in \mathbb{Z}$, where the subscripts are taken modulo $n$ and $m$, respectively. We call these edges \emph{parallel} edges to the edge $[x_i$, $y_j]$.

\begin{lemma} 
\label{2}
Let $G$ be a $2$-$\mathcal{M}$-closed multigraph, and let $C_1=(x_1$, \ldots, $x_n$, $x_1)$ and $C_2=(y_1$, \ldots, $y_m$, $y_1)$ be two alternating cycles in $G$ with an edge $[x_i$, $y_j]$ such that $c[x_i$, $x_{i+1}]=c[y_j$, $y_{j+1}]=c[x_i$, $y_j]$ (this is, $C_1$ and $C_2$ are appropriately labelled with respect to $[x_i$, $y_j]$). Then $[x_{i+1}$, $y_{j+1}]$, $[x_i$, $y_{j+1}]$ and $[y_{j}$, $x_{i+1}]$ are edges of $G$ with $c[x_{i+1}$, $y_{j+1}] \neq c[x_{i}$, $y_{j}]$, or there exists an alternating cycle with vertex set $V(C_1) \cup V(C_2)$.
\end{lemma}

\begin{proof}
In what follows, recall that if we find a good pair of edges, we are done since we can merge both cycles by Proposition~\ref{prop:goodpair}. So, we suppose that there are no good pairs of edges.

Assume without loss of generality that $C_1$ and $C_2$ are appropriately labelled with respect to $[x_1$, $y_1]$ and this is a blue edge. Notice that $(x_{2}$, $x_1$, $y_1)$ and $(x_1$, $y_1$, $y_{2})$ are monochromatic $2$-paths, hence $[x_{1}$, $y_{2}]$, $[x_{2}$, $y_{1}] \in E(G)$. If one of those edges is blue, then $(x_{2}$, $x_{1}$, $y_{2})$ or $(y_{2}$, $y_1$, $x_{2})$ is a monochromatic path, and therefore $[x_{2}$, $y_{2}] \in E(G)$.  If $[x_{2}$, $y_{2}]$ is blue, we are done; otherwise, it is a red edge and $[x_{1}$, $x_{2}]$, $[y_{1}$, $y_{2}]$ is a good pair of edges.  So assume that $[x_{1}$, $y_{2}]$ and $[x_{2}$, $y_{1}]$ are both red (see Figure \ref{figure4}).
\begin{figure}
\begin{minipage}{1\linewidth}
\begin{minipage}{0.45\linewidth}
 \includegraphics[height=3cm]{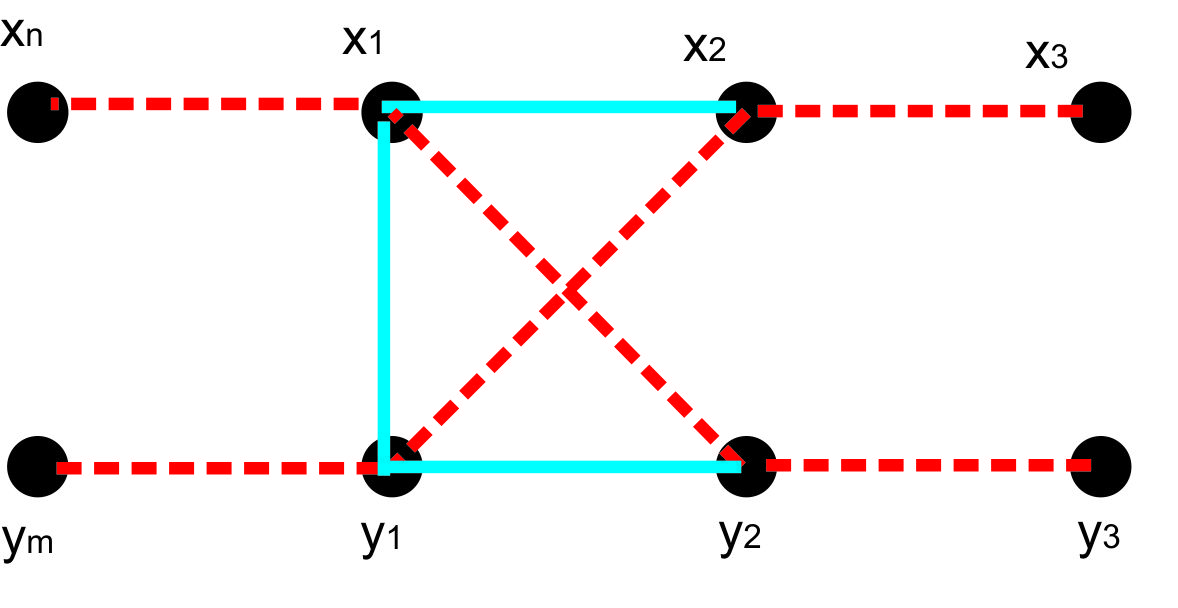}
 \caption{Initial edges in a $2$-$\mathcal{M}$-closed multigraph.}
 \label{figure4}
\end{minipage}
\hfill
\begin{minipage}{0.45\linewidth}
 \includegraphics[height=3cm]{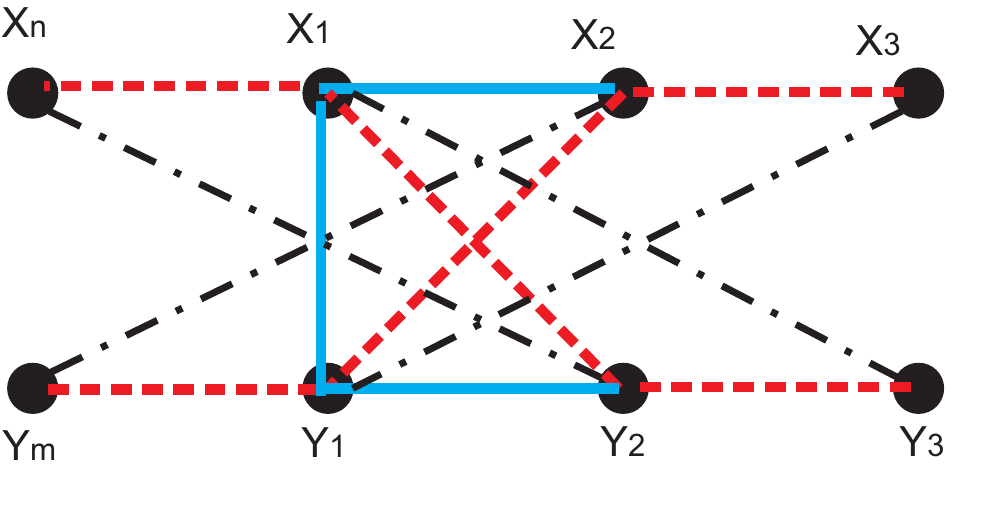}
 \caption{Structure of a $2$-$\mathcal{M}$-closed multigraph (the dash-and-dotted edges have unknown color.)}
 \label{figure4b}
\end{minipage}
\end{minipage}
\end{figure}

Since $(x_{3}$, $x_{2}$, $y_{1})$ and $(y_{3}$, $y_{2}$, $x_1)$ are monochromatic, we have that $[x_{1}$, $y_{3}]$, $[x_{3}$, $y_{1}] \in E(G)$. Similarly, $[x_n$, $y_2]$, $[y_m$, $x_2]\in E(G)$. See Figures~\ref{figure4b} and~\ref{figure5}. Hence, in each of the following cases, either these conditions are met, or else there is a good pair of edges and an alternating cycle comprising both sets of vertices. 

\vspace{\topsep}
\noindent\textit{Case 1.} $[x_{1}$, $y_{3}]$ and $[x_{3}$, $y_{1}]$ are red, thus $[x_n$, $y_3]$, $[y_m$, $x_3]\in E(G)$.
\begin{figure}
  \centerline{\includegraphics[height=3cm]{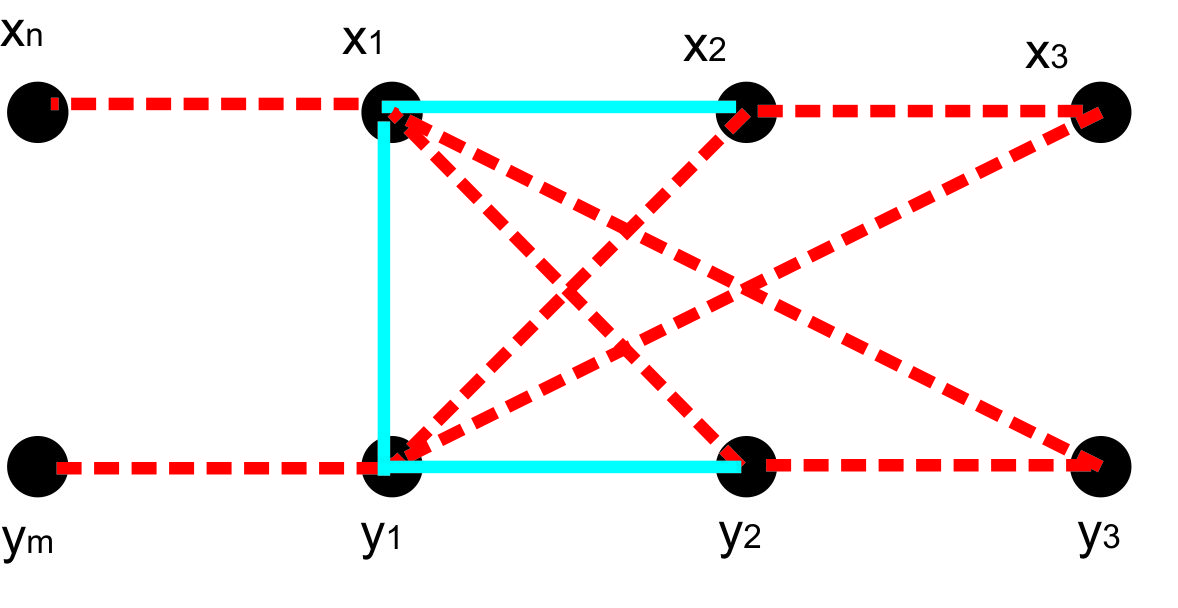}}
  \caption{Proof of Lemma~\ref{2}.}
\label{figure5}
\end{figure}

If one of the edges $[x_n$, $y_2]$, $[x_n$, $y_3]$, $[y_m$, $x_2]$, and $[y_m$, $x_3]$ is red, there exist a good pair of edges in $G$ (see Figures~\ref{figure6} and~\ref{figure7}).
\begin{figure}
\begin{minipage}{1\linewidth}
\begin{minipage}{0.45\linewidth}
 \includegraphics[height=3cm]{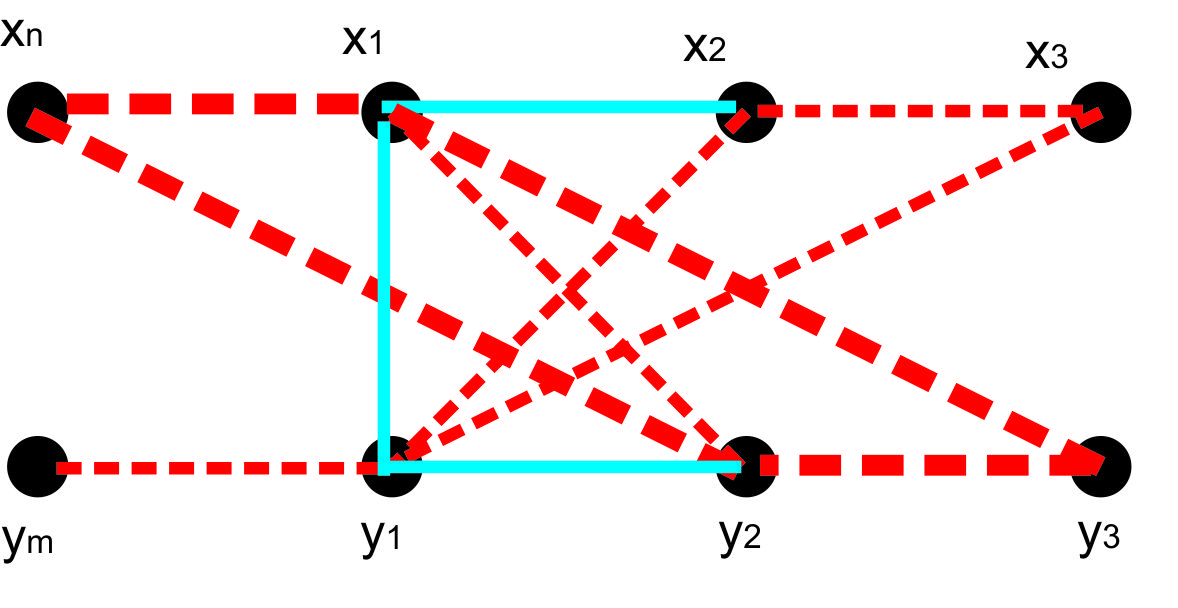}
 \caption{A good pair of edges in Case 1, when $[x_n,y_2]$ is red.}
 \label{figure6}
\end{minipage}
\hfill
\begin{minipage}{0.45\linewidth}
 \includegraphics[height=3cm]{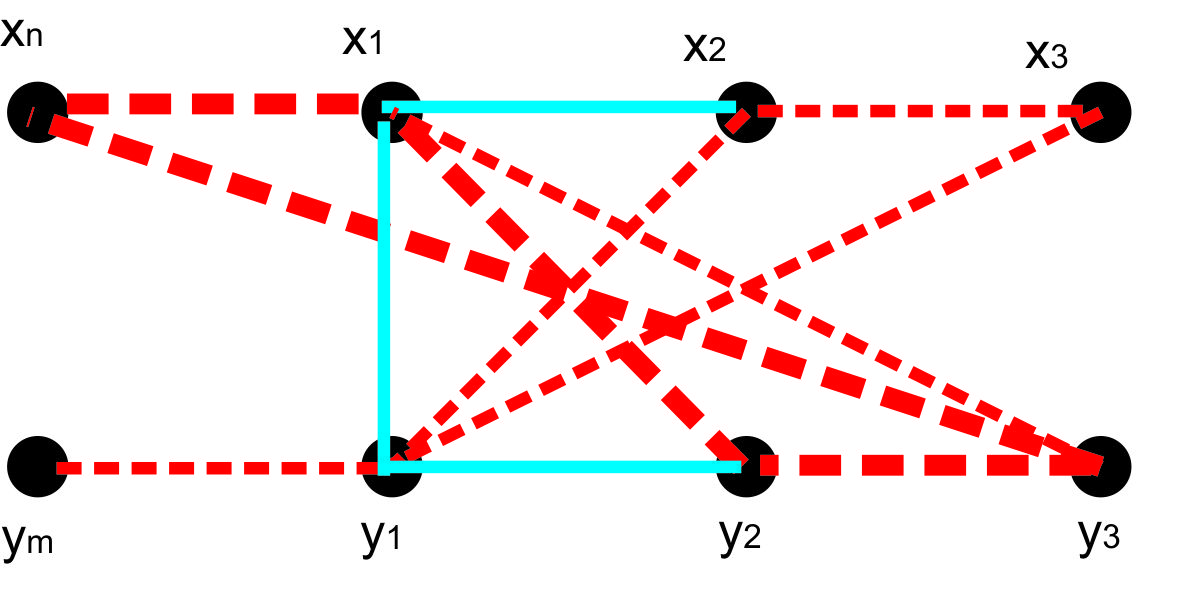}
 \caption{Other good pair of edges in Case 1 when $[x_n,y_3]$ is red.}
 \label{figure7}
\end{minipage}
\end{minipage}
\end{figure}

So assume that $[x_n$, $y_2]$, $[x_n$, $y_3]$, $[y_m$, $x_2]$ and $[y_m$, $x_3]$ are blue (see Figure \ref{figure8}).
 \begin{figure}
 \centering
 \includegraphics[height=2.9cm]{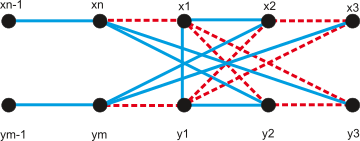}
 \caption{$[x_n$, $y_2]$, $[x_n$, $y_3]$, $[y_m$, $x_2]$ and $[y_m$, $x_3]$ are blue edges.}
 \label{figure8}
 \end{figure}
  \begin{figure}
 \centering
 \includegraphics[height=3cm]{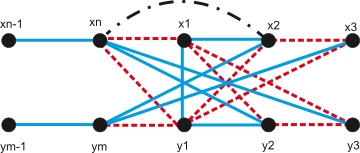}
 \caption{$[x_n$, $y_1]$ is a red edge.}
 \label{figure9}
\end{figure}
Therefore  $(x_n$, $y_2$, $y_1)$ is monochromatic, and hence $[x_n$, $y_1] \in E(G)$. We will show that no matter what is the color of the edge $[x_n$, $y_1]$, one of the edges $[x_2$, $y_2]$ and $[x_n$, $x_2]$ is in $G$. First assume that $[x_n$, $y_1]$ is red. Since $(x_n$, $y_1$, $x_2)$ is monochromatic, we  have that $[x_n$, $x_2]$ is in $G$ (see Figure \ref{figure9}).

Hence, assume that $[x_n$, $y_1]$ is blue. Therefore $[x_{n-1}$, $y_2]$ and $[x_{n-1}$, $y_1]$ are in $G$, since $(x_n$, $y_1$, $y_2)$ and $(x_{n-1}$, $x_n$, $y_1)$ are monochromatic and these edges are red (otherwise there exist a good pair of edges). Therefore  $(x_2$, $y_1$, $x_{n-1})$ is a monochromatic path and $[x_{n-1}$, $x_2]$ is in $G$. If $[x_{n-1}$, $x_2]$ is red, then $(y_2$, $x_{n-1}$, $x_2)$ is a monochromatic path, so $[x_2$, $y_2]\in E(G)$. If $[x_2$, $y_2]$ is a blue edge, then $[x_1$, $y_1]$ and $[x_2$, $y_2]$ is a good pair of edges and we can merge both cycles. Otherwise $[x_{n-1}$, $x_2]$ is blue, therefore $(x_n$, $x_{n-1}$, $x_2)$ is monochromatic and $[x_n$, $x_2]$ is in $G$. See Figure \ref{figure10}.
\begin{figure}
\begin{minipage}{1\linewidth}
\begin{minipage}{0.47\linewidth}
 \includegraphics[height=3cm]{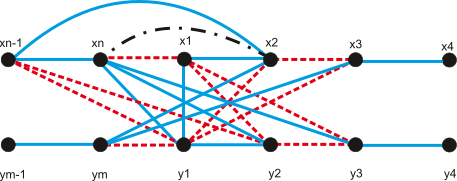}
 \caption{$[x_n$, $y_1]$ is a blue edge.}
 \label{figure10}
\end{minipage}
\hfill
\begin{minipage}{0.47\linewidth}
 \includegraphics[height=2.7cm]{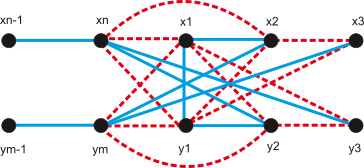}
 \caption{Alternating cycle in Case 1.}
 \label{figure11}
\end{minipage}
\end{minipage}
\end{figure}
Analogously, no matters the color of the edge $[y_m$, $x_1]$, we have $[y_m$, $y_2]\in E(G)$.

Thus, the edges $[y_m$, $y_2]$, $[x_n$, $x_2]$ are in $G$ and both are red (otherwise $(y_2$, $x_n$, $x_2)$ or $(x_2$, $y_m$, $y_2)$ are monochromatic, so $[x_2$, $y_2]$ will be in $G$). Therefore $(y_m$, $y_2$, $y_1$, $x_3$, $x_4$, \ldots, $x_n$, $x_2$, $x_1$, $y_3$, \ldots, $y_m)$ is an alternating cycle with vertex set $V(C_1) \cup V(C_2)$. See Figure \ref{figure11}.

\vspace{\itemsep}
\noindent\textit{Case 2.}  $[x_{1}$, $y_{3}]$ and $[x_{3}$, $y_{1}]$ are not both red. Suppose without loss of generality that $[x_{3}$, $y_{1}]$ is blue. Therefore $(y_2$, $y_1$, $x_3)$ is monochromatic, and $[x_3$, $y_2]\in E(G)$. See Figure \ref{fig12}.
\begin{figure}
  \centerline{\includegraphics[height=3cm]{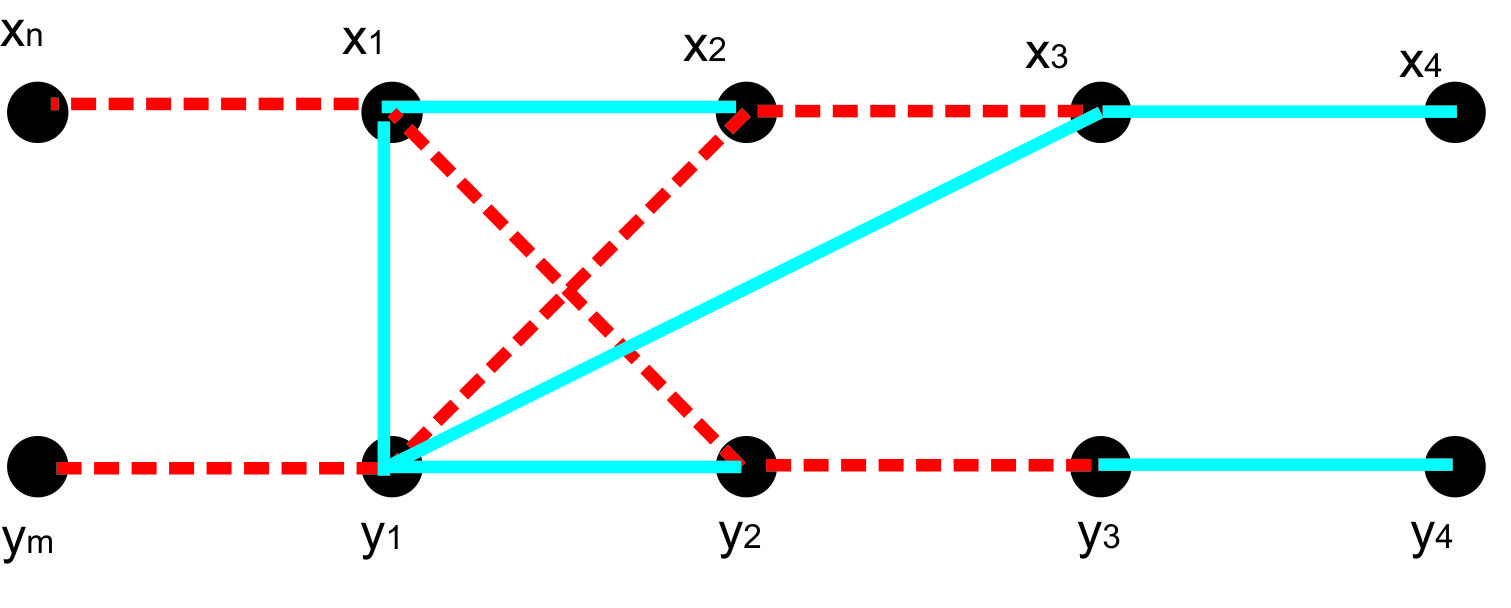}}
  \caption{Case 2.}
\label{fig12}
\end{figure}
If $[x_3$, $y_2]$ is red, then $(y_2$, $x_3$, $x_2)$ is monochromatic, and $[x_2$, $y_2]\in E(G)$. Otherwise, $[x_3$, $y_2]$ is blue (see Figure \ref{fig13}).
\begin{figure}
\begin{minipage}{1\linewidth}
\begin{minipage}{0.47\linewidth}
 \includegraphics[height=2.7cm]{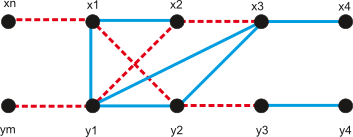}
  \caption{$[x_3$, $y_2]$ is a blue edge.}
\label{fig13}
\end{minipage}
\hfill
\begin{minipage}{0.47\linewidth}
  \includegraphics[height=3cm]{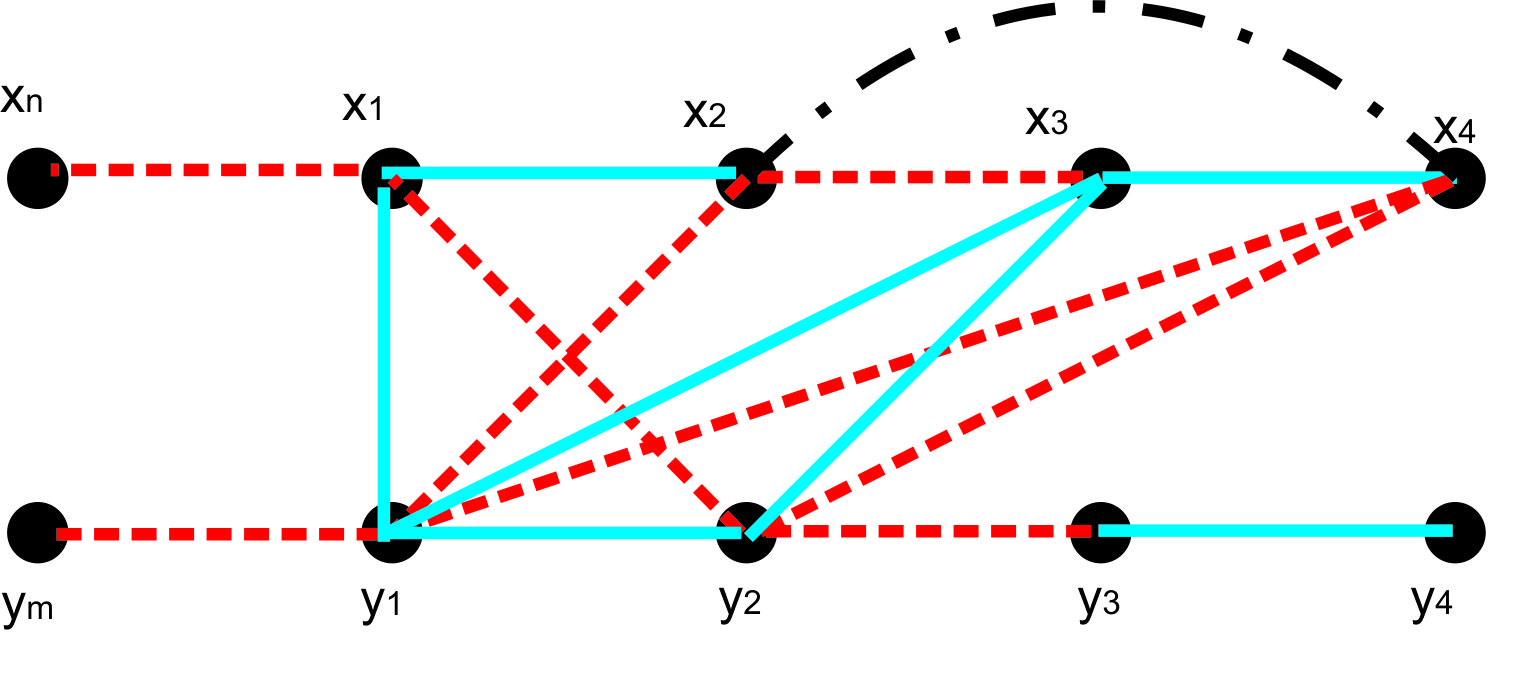}
  \caption{Edge $[x_2$, $x_4]$.}
\label{fig14}
\end{minipage}
\end{minipage}
\end{figure}

Therefore, $[y_1$, $x_4]$ is in $G$. If it is blue, then there exist a good pair of edges $(y_1$, $x_4$, $x_3$, $y_2$, $y_1)$. Hence, $[y_1$, $x_4]$ is red. The path $(x_2$, $y_1$, $x_4)$ is monochromatic, thus $[x_2$, $x_4] \in E(G)$. Then $(x_4$, $x_3$, $y_2)$ is also monochromatic, so $[x_4$, $y_2]\in E(G)$. If $[x_4$, $y_2]$ is blue, we have a good pair of edges. Hence, it is red (see Figure \ref{fig14}).

If $[x_2$, $x_4]$ is red, then $(x_2$, $x_4$, $y_2)$ is a monochromatic path, $[x_2$, $y_2]\in E(G)$ and we are done. Therefore, $[x_2$, $x_4]$ is blue and $(x_4$, $x_2$, $x_3$, $y_2$, $y_3$, \ldots, $y_1$, $x_1$, $x_n$, \ldots, $x_4)$ is an alternating cycle with vertex set $V(C_1) \cup V(C_2)$.

In all these cases, we have that the edge $[x_2$, $y_2]$ is in $G$ or there exists an alternating cycle with vertex set $V(C_1) \cup V(C_2)$.
\end{proof}

\begin{corollary} \label{3}
Let $C_1=(x_1$, \ldots, $x_n$, $x_1)$ and $C_2=(y_1$, \ldots, $y_m$, $y_1)$ be two alternating cycles in a $2$-$\mathcal{M}$-closed graph. If there exists an edge between $C_1$ and $C_2$, then at least one of the following holds:
\begin{enumerate}
 \item There exists an alternating cycle with vertex set $V(C_1) \cup V(C_2)$.
 \item $[x$, $y] \in E(G)$ for every $x \in C_1$ and $y \in C_2$.
\end{enumerate}
\end{corollary}

\begin{proof}
 We can assume that between $C_1$ and $C_2$ there is no good pair of edges (otherwise, we are done since we can merge both cycles by Proposition~\ref{prop:goodpair}). Notice that Lemma~\ref{2} states, in this situation, that if $(x_{i-1}, x_{i}, y_{j}, y_{j+1})$ is a monochromatic path between $C_{1}$ and $C_{2}$, then $[x_{i-1},y_{j+1}]$ exists and $c[x_{i-1},y_{j+1}]\neq c[x_{i}, y_{j}]$.
 
Suppose that $[x_1,y_j]$ is an edge in $G$, we will prove that $[x_1,y_{j+1}]$ is an edge in $G$. Since we can interchange $C_1$ and $C_2$ and repeat this procedure, we will obtain that $[x,y]\in E(G)$ for every $x \in C_1$ and $y \in C_2$. We consider the following cases:

\vspace{\topsep}
\noindent\textit{Case 1.} $c[x_1,y_j]=c[y_j,y_{j+1}]$. Then $(x_1,y_j,y_{j+1})$ is a monochromatic path. Therefore $[x_1,y_{j+1}]$ is an edge of $G$.

\vspace{\itemsep}
\noindent\textit{Case 2.} $c[x_1,y_j] \neq c[y_j,y_{j+1}]$. Suppose without loss of generality that $c[x_1,x_2]$ is blue.

\vspace{\itemsep}
\noindent\textit{Case 2.1} $c[x_1,y_j]$ is red and $c[y_j,y_{j+1}]$ is blue. Then $c[x_1,x_n]$ and $c[y_j,y_{j-1}]$ are red edges. Since $(x_{n}, x_{1}, y_{j}, y_{j-1})$ is a (monochromatic) red path, Lemma~\ref{2} implies that $[x_n,y_{j-1}]$ is a blue edge in $G$.  Now,  $(x_{n-1}, x_{n}, y_{j-1}, y_{j-2})$ is a (monochromatic) blue path, so we have that  $[x_{n-1},y_{j-2}]$ is a red edge in $G$. Continuing this way, we obtain that $[x_{1-i},y_{j-i}]$ is an edge in $G$ for every $i\in \mathbb{N}$, where the subscripts are taken modulo $n$ and $m$, respectively.  Applying Lemma~\ref{2} to edges $[x_{1-(i+1)}, x_{1-i}]$, $[x_{1-i},y_{j-i}]$ and $[y_{j-(i+1)},y_{j-i}]$, we have that $[x_{1-(i+1)},y_{j-i}] \in E(G)$. Therefore for $i=nm-1$, $[x_{1-((nm-1)+1)},y_{j-(nm-1)}]$ is in $G$, but $[x_{1-((nm-1)+1)},y_{j-(nm-1)}]=[x_{1},y_{j+1}]$ since the subscripts are taken modulo $n$ and $m$, respectively.

\vspace{\itemsep}
\noindent\textit{Case 2.2} $c[x_1,y_j]$ is blue and $c[y_j,y_{j+1}]$ is red. Then $c[y_j,y_{j-1}]$ is blue. Since $(x_{2}, x_{1}, y_{j}, y_{j-1})$ is a (monochromatic) blue path, Lemma~\ref{2} implies that $[x_2,y_{j-1}]$ is a red edge. Now, $(x_{3}, x_{2}, y_{j-1}, y_{j-2})$ is a (monochromatic) red path, so we have that $[y_{j-1},y_{j-2}]$, $[x_3,y_{j-2}]$ is a blue edge. Continuing this way, we obtain that $[x_{1+i},y_{j-i}]$ is an edge in $G$ for every $i\in \mathbb{N}$, where the subscripts are taken modulo $n$ and $m$, respectively. Applying Lemma~\ref{2} to edges $[x_{1+(i+1)}, x_{1+i}]$, $[x_{1+i},y_{j-i}]$ and $[y_{j-(i+1)},y_{j-i}]$, we have that $[x_{1+(i+1)},y_{j-i}] \in E(G)$. Therefore for $i=nm-1$, $[x_{1+(nm-1)+1},y_{j-(nm-1)}]$ is in $G$, but $[x_{1+(nm-1)+1},y_{j-(nm-1)}]=[x_{1},y_{j+1}]$ since the subscripts are taken modulo $n$ and $m$, respectively.
\end{proof}

Let $C$ be a cycle in $G$. We denote by {\bf $I_{C}$} (resp. {\bf $P_{C}$}) the set of vertices with odd (resp. even) subscript in $C$.
\begin{definition}
Let $C_1$ and $C_2$ be two disjoint alternating cycles in a $2$-$\mathcal{M}$-closed graph $G$. We will say that \emph{$C_1$ blue-dominates} (resp. \emph{red-dominates) $C_2$} whenever the following conditions are satisfied:
\begin{itemize}
\item $[x$, $y]\in E(G)$ for every $x \in C_1$ and $y\in C_2$.
\item $G[I_{C_1}]$ and $G[P_{C_1}]$ are complete graphs.
\item All the edges in $G[I_{C_1}]$ are blue (resp. red).
\item All the edges in $G[P_{C_1}]$ are red (resp. blue).
\item All the edges between $I_{C_1}$ and $C_2$ are blue (resp. red).
\item All the edges between $P_{C_1}$ and $C_2$ are red (resp. blue).
\end{itemize}
\begin{figure}
  \centerline{\includegraphics[height=4.5cm]{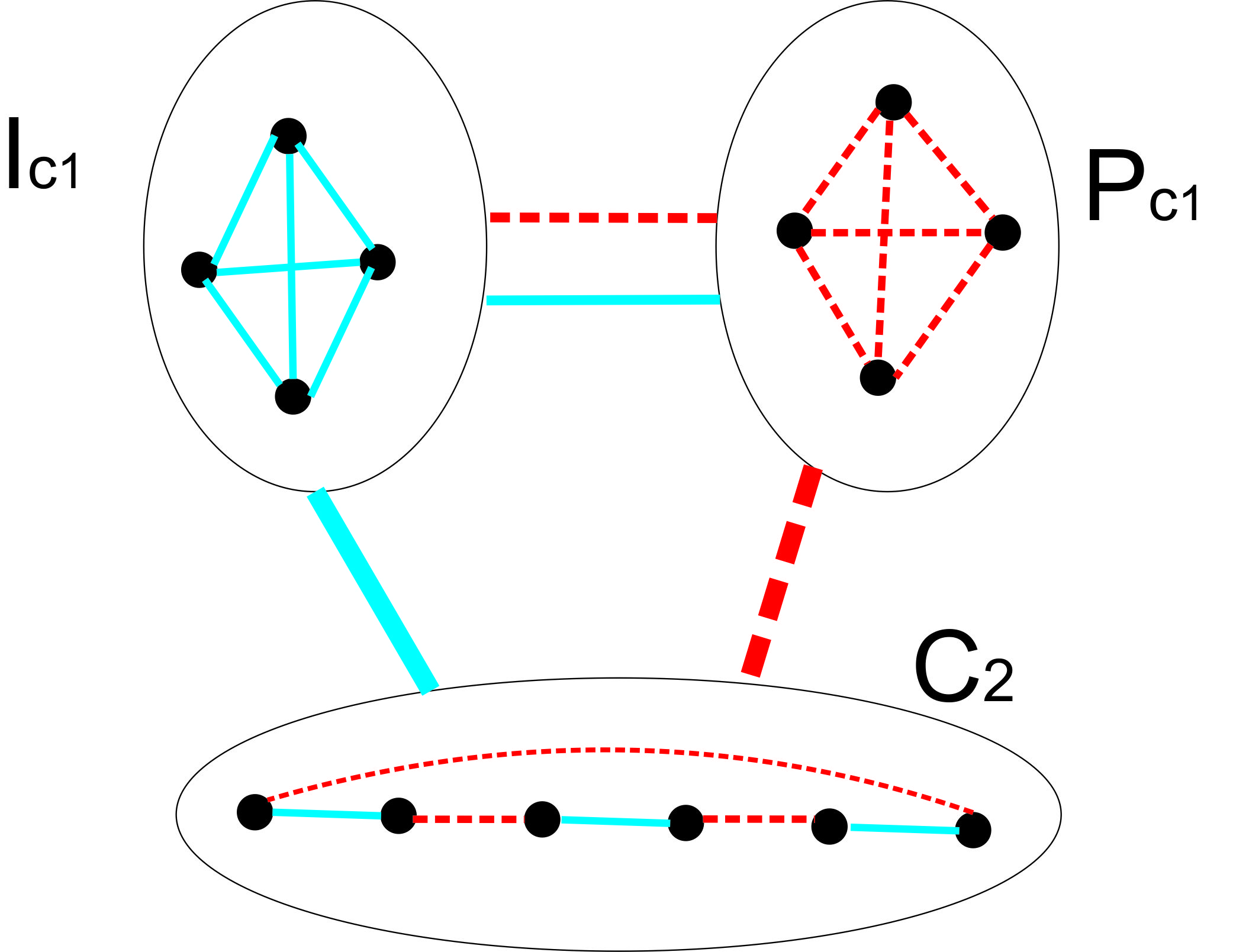}}
  \caption{$C_1$ blue-dominates $C_2$.}
\label{fig23}
\end{figure}
See Figure~\ref{fig23}. Whenever $C_1$ blue-dominates or red-dominates $C_2$, we say that $C_1$ \emph{color-dominates} $C_2$.
\end{definition}

We will make use of some concepts about digraphs. The vertex set and the arc set of a digraph $D$ are denoted by $V(D)$ and $A(D)$, respectively. A digraph for which there exists exactly one arc between each pair of vertices is a \emph{tournament}. An arc $(u,v)\in A(D)$ is \emph{symmetric} if $(v,u)\in A(D)$; otherwise, it is \emph{asymmetric}. A digraph $D$ is \emph{quasi-transitive} if $u$ and $w$ are adjacent whenever $(u$, $v)$, $(v$, $w)\in A(D)$ and it is a \emph{local tournament} if there is exactly one arc between $u$ and $w$ whenever $(u$, $v)$, $(w$, $v)\in A(D)$ or $(v$, $u)$, $(v$, $w)\in A(D)$. For digraphs with no symmetric arcs, the intersection of quasi-transitive digraphs and local tournaments is precisely the class of tournaments (see Chapter~2 of~\cite{BanGut2009}). A digraph $D$ is \emph{acyclic} if it has no directed cycles.

Now, we will prove the main theorem.
\begin{theorem}
Let $G$ be a $2$-$\mathcal{M}$-closed graph. $G$ has an alternating Hamiltonian cycle if and only if it is color-connected and has an alternating cycle factor.
\end{theorem}

\begin{proof}
If $G$ is a $2$-$\mathcal{M}$-closed graph and has an alternating Hamiltonian cycle, then $G$ is color-connected and has an alternating cycle factor. So, let $G$ be a $2$-$\mathcal{M}$-closed and color-connected graph. Given any alternating cycle factor of $G$ with at least two cycles, $C=\{C_1$, \ldots, $C_l \}$, we will construct a smaller alternating cycle factor. Since $G$ is connected, there exist two alternating cycles $C_1=(x_1$, \ldots, $x_n$, $x_1)$ and $C_2=(y_1$, \ldots, $y_m$, $y_1)$ with an edge between them. 
Suppose without loss of generality that $[x_1$, $y_1] \in E(G)$, it is blue and that $C_1$ and $C_2$ are appropriately labelled with respect to $[x_1$, $y_1]$. Proposition~\ref{prop:goodpair} implies that if there exists a good pair of edges, then we can merge both cycles, obtaining a smaller alternating cycle factor. So, we can assume that there is no good pair of edges and Corollary~\ref{3} implies that $[x$, $y] \in E(G)$ for every $x \in C_1$ and $y \in C_2$. Also, this implies that consecutive parallel edges between $C_1$ and $C_2$ have different color. Now, we analyze the color structure of the set of edges between $x_1$ and $C_2$. 

\vspace{\topsep}
\noindent\textit{Case 1.} There are blue and red edges between $x_1$ and $I_{C_2}$. We can suppose that $[x_1$, $x_2]$, $[y_1$, $y_2]$, $[x_1$, $y_1]$ are blue, and $[x_1$, $y_3]$ is red. Since parallel edges between $C_1$ and $C_2$ have alternating colors (otherwise there exists a good pair of edges), we have that
\begin{displaymath}
 (y_1, y_2, x_2, x_1, y_3, y_4, x_4, x_3, \ldots, y_n, x_n, x_{n-1}, y_{n+1}, y_{n+2}, \ldots, y_m, y_1)\text{,}
\end{displaymath}
when $n \leq m$, or
\begin{multline*}
(y_1, y_2, x_2, x_1, y_3, y_4, x_4, x_3, \ldots, y_{m-2}, x_{m-2}, x_{m-3}, y_{m-1}, y_m, x_n, x_{n-1}, \ldots, x_m, x_{m-1}, y_1)\text{,} 
\end{multline*}
when $n > m$ (notice that arc $[x_{m-1}$, $y_1]$ is parallel to $[x_{1}$, $y_3]$), is an alternating cycle with vertex set $V(C_1) \cup V(C_2)$ (see Figure~\ref{fig15}).
\begin{figure}
  \centerline{\includegraphics[height=2.7cm]{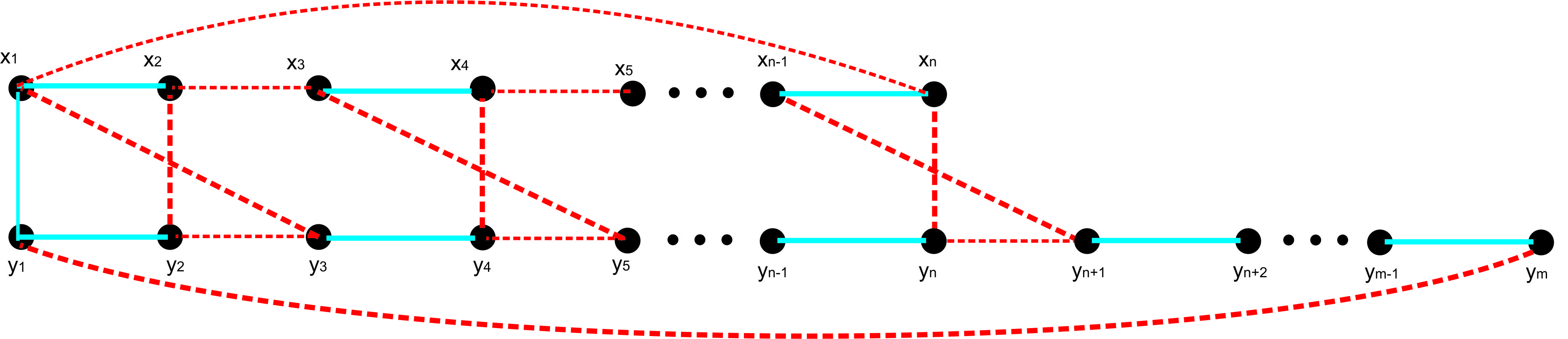}}
  \caption{Case 1.}
\label{fig15}
\end{figure}
 
\vspace{\itemsep}
\noindent\textit{Case 2.} There are blue and red edges between $x_1$ and $P_{C_2}$. Without loss of generality, assume that $[x_{1}$, $y_{k}]$ is a red edge and $[x_{1}$, $y_{k+2}]$ is a blue edge, for some even integer $k$. Recall that $[x_{1}$, $x_{2}]$ and $[y_{k+1}$, $y_{k+2}]$ are blue edges. Since parallel edges between $C_1$ and $C_2$ have alternating colors, we have that $[x_{2}$, $y_{k+1}]$ is a blue edge. Therefore, $[x_{1}$, $x_{2}]$, $[y_{k+1}$, $y_{k+2}]$ is a good pair of edges.

\vspace{\itemsep}
\noindent\textit{Case 3.} All the edges between $x_1$ and $I_{C_2}$ and all the edges between $x_1$ and $P_{C_2}$ have the same color. First assume that the edges between $x_1$ and $C_2$ are not monochromatic, this implies that the edges between $x_1$ and $I_{C_2}$ are blue and the edges between $x_1$ and $P_{C_2}$ are red (see Figure~\ref{fig16b}). Since the colors of consecutive parallel edges between $C_1$ and $C_2$ are alternating, we have that all the edges between $x_2$ and $I_{C_2}$ are blue and the edges between $x_2$ and $P_{C_2}$ are red. And so on. Therefore, all the edges between $y_1$ and $C_1$ are blue. Interchanging $C_1$ and $C_2$, we can always assume that all the edges between $x_1$ and $C_2$ have the same color.
\begin{figure}
  \centerline{\includegraphics[height=3.5cm]{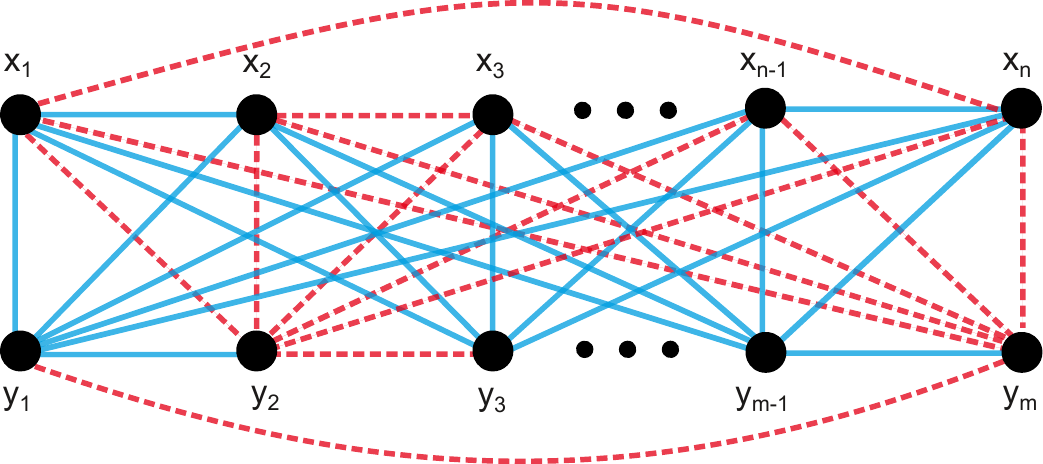}}
  \caption{Case 3, first assumption.}
\label{fig16b}
\end{figure}

Suppose without loss of generality that all the edges between $x_1$ and $C_2$ are blue (see Figure~\ref{fig16}). Since the colors of the parallel edges between $C_1$ and $C_2$ are alternating, the edges between $x_2$ and $C_2$ are all red. Therefore all the edges between $P_{C_1}$ and $C_2$ are red and all the edges between $I_{C_1}$ and $C_2$ are blue.
\begin{figure}
  \centerline{\includegraphics[height=3.5cm]{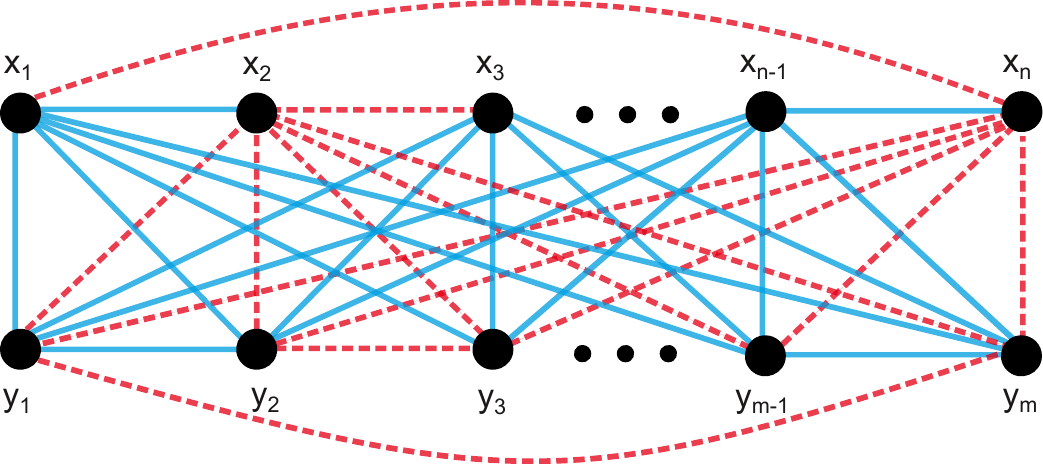}}
  \caption{Case 3, second assumption.}
\label{fig16}
\end{figure}

These conditions fix the edges between $C_1$ and $C_2$. Notice that $(x_{2i+1}$, $y_j$, $x_{2k+1})$ and $(x_{2i}$, $y_j$, $x_{2k})$ are monochromatic paths for every $i$, $k \in \mathbb{Z}$, where $i$ and $k$ are taken modulo $n$. Therefore, $[x_{2i+1}$, $x_{2k+1}]$, $[x_{2i}$, $x_{2k}] \in E(G)$. Now, we will analyze the structure inside the cycle $C_1$.

\vspace{\topsep}
\noindent\textit{Case 3.1} $[x_{2k_1+1}$, $x_{2k_2+1}]$ is a red edge or $[x_{2k_3}$, $x_{2k_4}]$ is a blue edge (recall that $x_{2k_1+1}$, $x_{2k_2+1} \in I_{C_1}$ and $x_{2k_3}$, $x_{2k_4} \in P_{C_1}$).

Suppose that $2k_1+1 < 2k_2+1$. If $[x_{2k_1+1}$, $x_{2k_2+1}]$ is a red edge, then there exist an alternating cycle merging both cycles $(y_1$, $x_{2k_1}$, $x_{2k_1-1}$, \ldots,~$x_{2k_2+1}$, $x_{2k_1+1}$, $x_{2k_1+2}$, \ldots,~$x_{2k_2}$, $y_m$, $y_{m-1}$, \ldots,~$y_1)$. Analogously if $[x_{2k_3}$, $x_{2k_4}]$ is blue.

\vspace{\itemsep}
\noindent\textit{Case 3.2} All the edges between vertices in $I_{C_1}$ are blue, and all the edges between vertices in $P_{C_1}$ are red. Therefore, $C_1$ blue-dominates $C_2$.

Notice that if $C_i$ does not color-dominate $C_j$ and $C_j$ does not color-dominate $C_i$, we can reduce the number of cycles in the alternating cycle factor. Therefore, in a minimum alternating cycle factor, $C_i$ color-dominates $C_j$ or $C_j$ color-dominates $C_i$ for every pair of adjacent cycles in such factor.

So, given an alternating cycle factor $C$, we can reduce the number of cycles in $C$ or between every pair $C_i$ and $C_j$ of adjacent cycles in $C$, $C_i$ color-dominates $C_j$ or $C_j$ color-dominates $C_i$.

Let $C=\{C_1$, \ldots, $C_l \}$ be a minimum alternating cycle factor of $G$, where $C_i=(x_{1,i}$, $x_{2,i}$, $x_{3,i}$, \ldots, $x_{m_{i},i}$, $x_{1,i})$, where $m_{i}$ is the length of $C_{i}$, for every $i \in \{1$, \ldots, $l\}$. We define the colored digraph $D_{G_C}^*$ as follows: 
\begin{itemize}
 \item $V(D_{G_C}^*)= C$, and
 \item $(C_i$, $C_j) \in A(D_{G_C}^*)$ if and only if $C_i$ color-dominates $C_j$. Moreover, $c[C_i$, $C_j]=c[x_{1,i}$, $x_{1,j}]$. This is, arcs have the same color as the color-domination between $C_i$ and $C_j$.
\end{itemize}

Let us suppose that $C_{1}$ blue-dominates $C_{2}$ and that the first edge of $C_{1}$ (\textit{i.e.}, $[x_{1,1}$, $x_{2,1}]$) is red. The cycle $C_{1}^{rev}=(x_{1,1}$, $x_{m_{i},1}$, \ldots, $x_{2,1}$, $x_{1,1})$ is the \emph{reverse} of $C_{1}$. Notice that $I_{C_{1}}=I_{C_{1}^{rev}}$ and $P_{C_{1}}=P_{C_{1}^{rev}}$. Hence, $C_{1}^{rev}$ also blue-dominates $C_{2}$ and the first edge of $C_{1}^{rev}$ is blue. Therefore, without loss of generality, we can always choose the color of the first edge of each cycle.

Now, we will prove that  $D_{G_C}^*$ is an acyclic tournament. Let $C'=(C_1$, $C_2$, \ldots, $C_p$, $C_1)$ be a cycle of minimum length in $V(D_{G_C}^*)$. If $p = 2$, then all  the edges between $x_{1,1}$ and $C_2$ are monochromatic, and all the edges between $x_{1,2}$ and $C_1$ are also monochromatic; moreover, all those edges have the same color. Hence, if $C_1$ blue-dominates $C_2$, then $C_2$ blue-dominates $C_1$. Assume without loss of generality that the first edges of $C_{1}$ and $C_{2}$ are blue. Therefore $[x_{1,1}$, $x_{2,1}]$ and $[x_{1,2}$, $x_{2,2}]$ is a good pair of edges between $C_1$ and $C_2$, so we can merge both alternating cycles and reduce the number of cycles in $C$, which is a contradiction. Thus, we can assume that $p \geq 3$. 

We will show that $p=3$. Assume that $p > 3$. Since $(C_1$, $C_2)$ and $(C_2$, $C_3)$ are arcs in $D_{G_C}^*$, we have that $C_1$ color-dominates $C_2$ and $C_2$ color-dominates $C_3$. Therefore, there exists a monochromatic path $(x_{a,1}$, $x_{b,2}$, $x_{c,3})$ in $G$ and we have an edge between $C_1$ and $C_3$. Moreover, there exists an arc between $C_1$ and $C_3$ in $D_{G_C}^*$. If $(C_3$, $C_1) \in A(D_{G_C}^*)$, we are done. So we have that $(C_1$, $C_3)$ is an arc in $A(D_{G_C}^*)$, which is a contradiction with the minimality of $C'$. Let $(C_1$, $C_2$, $C_3$, $C_1)$ be the minimum cycle in $D_{G_C}^*$. We consider two possibilities, either $(C_1$, $C_2$, $C_3$, $C_1)$ is monochromatic or not.

\vspace{\topsep}
\noindent\textit{Case 3.2.1} $(C_1$, $C_2$, $C_3$, $C_1)$ is monochromatic. Assume without loss of generality that $C_{i}$ blue-dominates $C_{i+1}$ and, moreover, that $[x_{1,i}$, $x_{2,i}]$ is a blue edge, where the subscripts are taken modulo $3$. Recall that all the edges between $x_{1,i}$ and $C_{i+1}$ are blue and all the edges between $x_{m_{3},3}$ and $C_{1}$ are red. Hence, $(x_{1,1}$, $x_{2,1}$, \ldots, $x_{m_{1},1}$, $x_{1,2}$, $x_{2,2}$, \ldots, $x_{m_{2},2}$, $x_{1,3}$, $x_{2,3}$, \ldots, $x_{m_{3},3}$, $x_{1,3}$, $x_{2,3}$, \ldots, $x_{m_{3},3}$, $x_{1,1})$ is an alternating cycle.

\vspace{\itemsep}
\noindent\textit{Case 3.2.2} $(C_1$, $C_2$, $C_3$, $C_1)$ is not monochromatic. Assume without loss of generality that $C_{1}$ blue-dominates $C_{2}$, $C_{2}$ blue-dominates $C_{3}$ and $C_{3}$ red-dominates $C_{1}$. Moreover, assume that $[x_{1,i}$, $x_{2,i}]$ is a blue edge, for $i \in \{1$, $2\}$ and a red edge, for $i = 3$. Recall that all the edges between $x_{1,i}$ and $C_{i+1}$ are blue, for $i \in \{1$, $2\}$ and all the edges between $x_{1,3}$ and $C_{1}$ are red. Hence, $(x_{1,1}$, $x_{2,1}$, \ldots, $x_{m_{1},1}$, $x_{1,2}$, $x_{2,2}$, \ldots, $x_{m_{2},2}$, $x_{m_{3},3}$, $x_{m_{3}-1,3}$, \ldots, $x_{2,3}$, $x_{1,3}$, $x_{1,1})$ is an alternating cycle.

\vspace{\itemsep}
In both cases, we obtain an alternating cycle with vertex set $V(C_1) \cup V(C_2) \cup V(C_3)$, which is a contradiction to the choice of $C$. Therefore, $D_{G_C}^*$ is an acyclic digraph. Moreover, the previous argument also shows that $D_{G_C}^*$ is a quasi-transitive digraph.

Now, let us show that for $C_i$, $C_j$ and $C_k$, if we have $(C_i$, $C_j)$ and $(C_i$, $C_k)$ are arcs in $D_{G_C}^*$. Then $c[C_i$, $C_j] = c[C_i$, $C_k]$. Assume on the contrary that $c[C_i$, $C_j] \neq c[C_i$, $C_k]$. Suppose without loss of generality that $C_i$ red-dominates $C_j$ and that $C_i$ blue-dominates $C_k$. The definition of domination implies that all the edges between vertices in $I_{C_i}$ are red and are also blue, which is a contradiction.

Let $(C_i$, $C_j)$ and $(C_i$, $C_k)$ be arcs in $D_{G_C}^*$. We know that $c[C_i$, $C_j] = c[C_i$, $C_k]$, so  $(x_{1,j}$, $x_{1,i}$, $x_{1,k})$ is monochromatic path. Therefore, $(C_j$, $C_k)$ or $(C_k$, $C_j)$ is in $A(D_{G_C}^*)$. Analogously, if $(C_j$, $C_i)$ and $(C_k$, $C_i)$ are arcs in $D_{G_C}^*$, we have that  $(C_j$, $C_k)$ or $(C_k$, $C_j)$ is in $A(D_{G_C}^*)$. Therefore $D_{G_C}^*$ is a local tournament. 

Since $G$ is a connected graph and $D_{G_C}^*$ is both a quasi-transitive digraph and a local tournament, we have that $D_{G_C}^*$ is an acyclic tournament (recall that there is no symmetric arcs in $D_{G_C}^*$). Therefore, there exists a vertex of out-degree $l-1$. Assume without loss of generality that this vertex is $C_1$, so we have $(C_1$, $C_i)\in A[D_{G_C}^*]$ for every $i\in \{2$, $3$, \ldots, $l\}$. Moreover, all these arcs have the same color, let it be blue. Therefore $C_1$ blue-dominates every other cycle in $D_{G_C}^*$, as it is depicted in Figure~\ref{NC}. Therefore, $G$ is not color-connected (since every vertex in $I_{C_1}$ does not have any alternating path starting with a red edge to any vertex in $V(G-C_1)$), which is a contradiction.
\begin{figure}
\begin{minipage}{1\linewidth}
\begin{minipage}{0.47\linewidth}
  \centering
  \includegraphics[height=5cm]{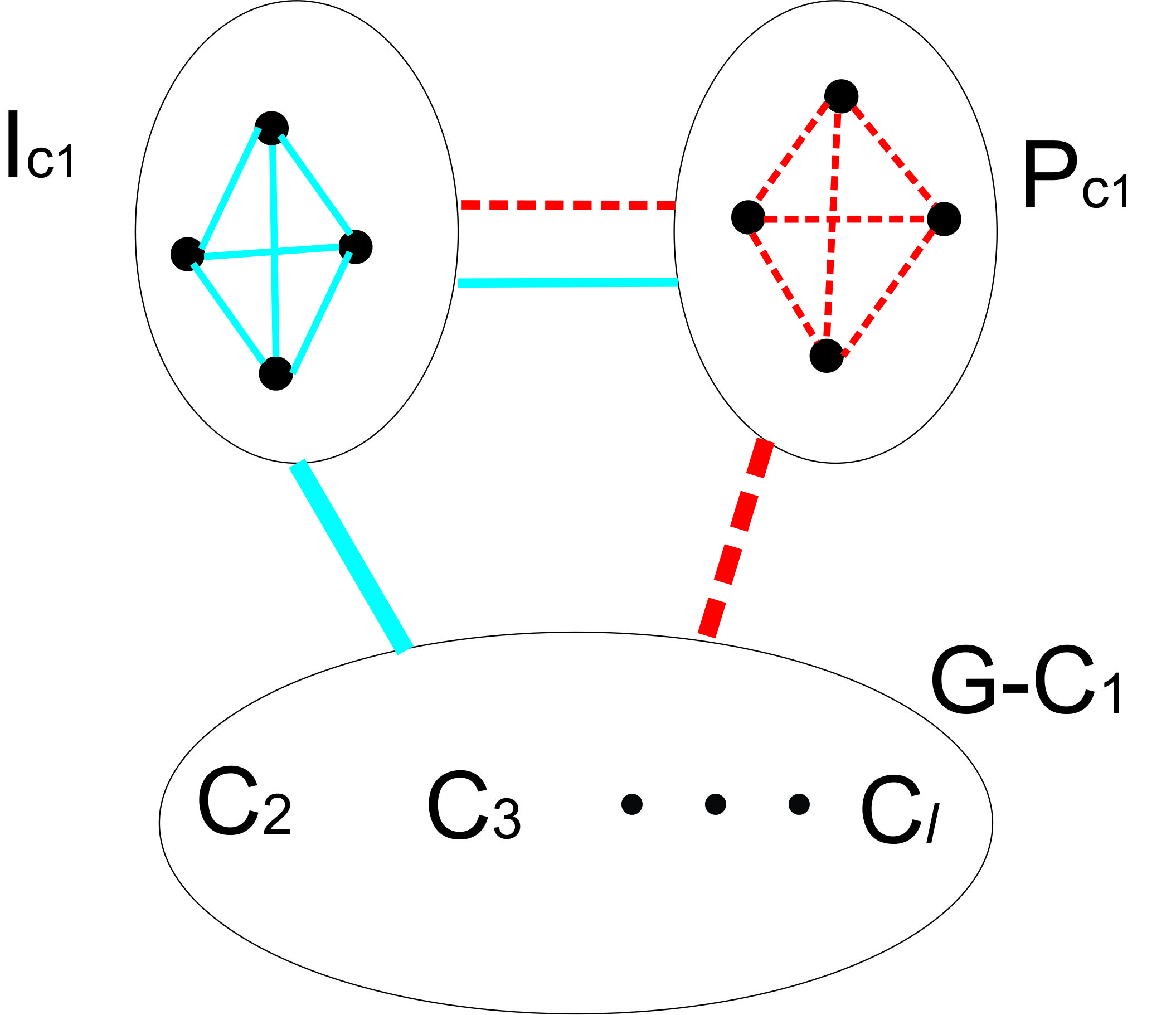}
  \caption{$G$ is not color-connected.}
\label{NC}
\end{minipage}
\hfill
\begin{minipage}{0.47\linewidth}
\centering
 \includegraphics[height=4cm]{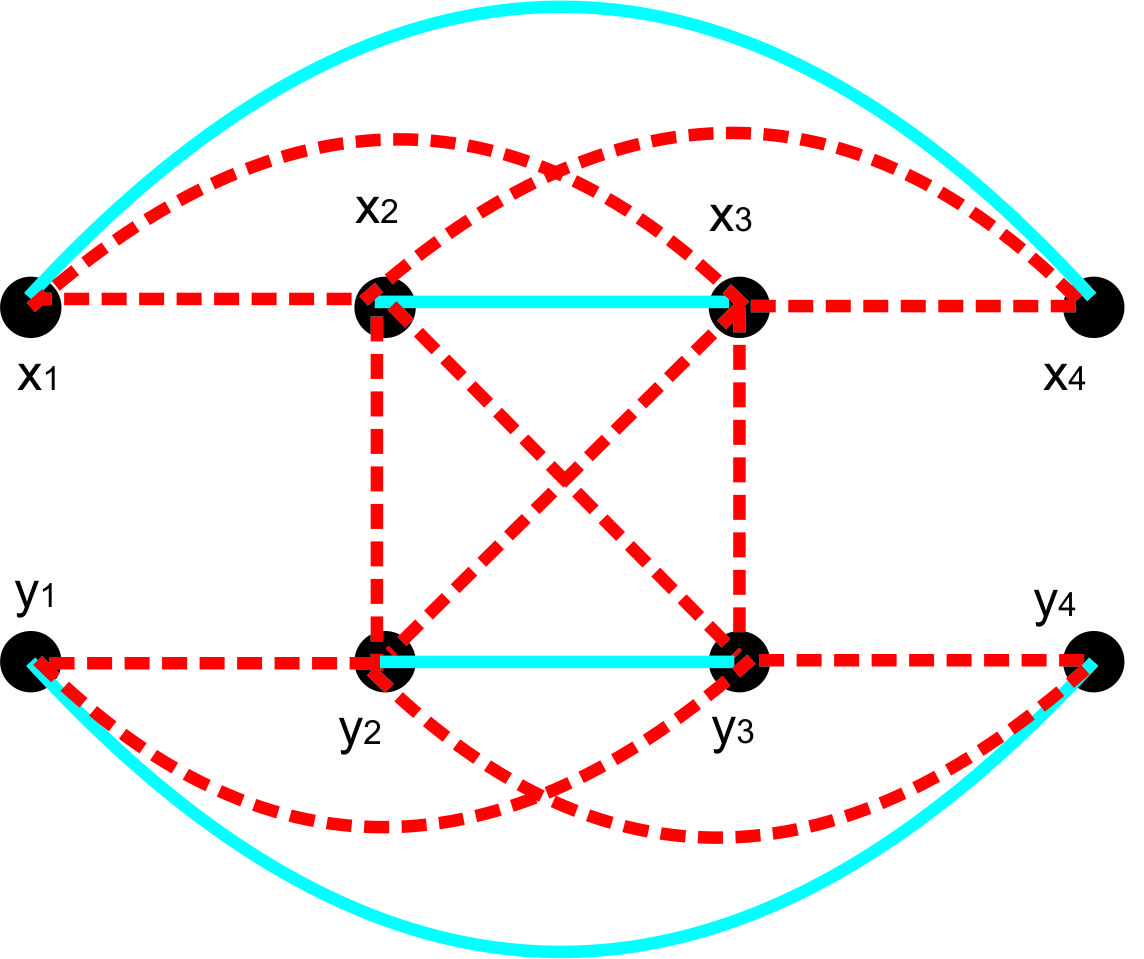}
 \caption{Counterexample.}
 \label{Counterexample}
\end{minipage}
\end{minipage}
\end{figure}
\end{proof}

To be color-connected and to have an alternating cycle factor are necessary conditions for a graph to have an alternating Hamiltonian cycle. Nevertheless, these conditions are insufficient for the $2$-$\mathcal{NM}$-closed graphs. In Figure~\ref{Counterexample}, we show a $2$-$\mathcal{NM}$-closed and color-connected graph,  with an alternating cycle factor and with no alternating Hamiltonian cycles at all.

Moreover, we can construct a family of such graphs. If we take two alternating cycles $C_1$, $C_2$ in which all the edges inside of them are red, and we ``paste'' the cycles taking $x_1$, $x_2 \in V(C_1)$ and $y_1$, $y_2 \in V(C_2)$ such that $[x_1$, $x_2]$, $[y_1$, $y_2]$ are blue, and we only add the red edges $[x_1$, $y_1]$, $[x_2$, $y_2]$, $[x_1$, $y_2]$ and $[x_2$, $y_1]$, then we obtain a graph which is color-connected and has alternating cycle factor, but that cannot be merged in order to obtain an alternating cycle (see Figure~\ref{Family}).
\begin{figure}
  \centerline{\includegraphics[height=4cm]{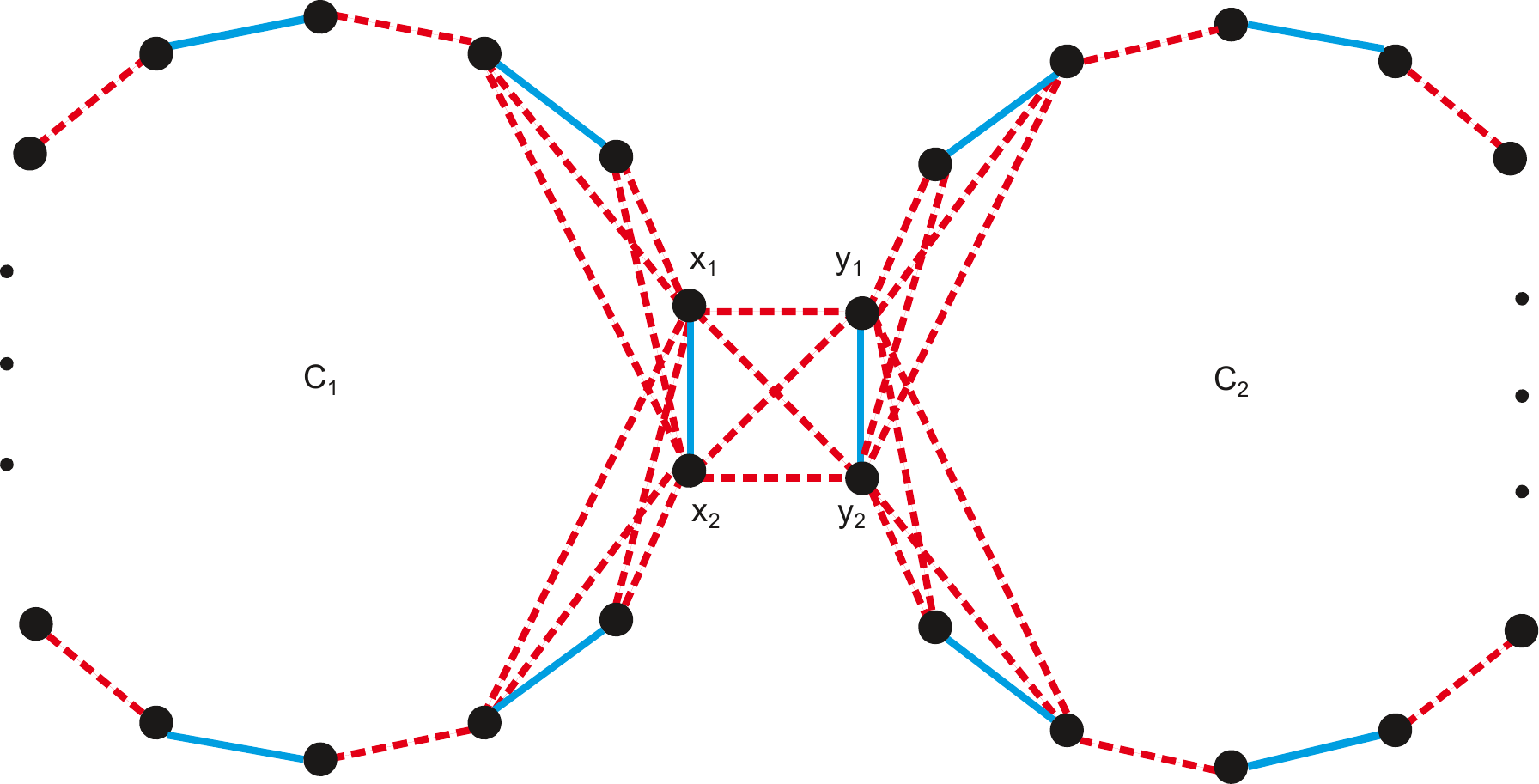}}
  \caption{Family with no alternating Hamiltonian cycles at all.}
\label{Family}
\end{figure}

\section*{Acknowledgments}
The authors are very grateful to the anonymous referees for a thorough review and their helpful suggestions which improved substantially the writing of this paper.

\bibliographystyle{abbrvnat}
\bibliography{bib-alternating}
\label{sec:biblio}

\end{document}